\documentclass[a4paper, 12pt, final]{amsart}
\usepackage{amscd}
\usepackage{mathrsfs}
\usepackage{}
\usepackage[all]{xy}
\usepackage[OT1]{fontenc}
\usepackage[latin1]{inputenc}
\usepackage{amssymb,latexsym}
\usepackage{type1cm,ae}
\usepackage[notref,notcite]{showkeys}
\usepackage{graphicx}
\usepackage{xspace}
\usepackage{setspace}
\usepackage[english]{babel}

\theoremstyle{definition}
		\newtheorem{theorem}{Theorem}[section]
				\newtheorem{proposition}[theorem]{Proposition}

     	        \newtheorem{definition}[theorem]{Definition}
	            \newtheorem{remark}[theorem]{Remark}
\numberwithin{equation}{section}

\newcommand*{\bR}{\ensuremath{\mathbb{R}}}

\newcommand*{\Wert}{\mathord{\mbox{|\kern-1.5pt|\kern-1.5pt|}}}
\newcommand*{\ie}{\mbox{i.e.}\xspace}

\DeclareMathOperator{\Lip}{Lip}

\def\XXint#1#2#3{{\setbox0=\hbox{$#1{#2#3}{\int}$}
  \vcenter{\hbox{$#2#3$}}\kern-.5\wd0}}

\title[Intrinsic geometry and analysis of Finsler structures]{Intrinsic geometry and analysis of Finsler structures}
\author{Chang-Yu Guo}

\address[Chang-Yu Guo]{Department of Mathematics and Statistics, University of Jyv\"askyl\"a, P.O. Box 35, FI-40014, Jyv\"askyl\"a, Finland and Department of Mathematics, University of Fribourg, CH-1700, Fribourg, Switzerland}
\email{changyu.guo@unifr.ch}

\thanks{C.Y.Guo was supported by the Magnus Ehrnrooth foundation.}

\begin{document}
\maketitle

\begin{abstract}
In this short note, we prove that if $F$ is a weak upper semicontinuous admissible Finsler structure on a domain in $\bR^n$, $n\geq 2$, then the intrinsic distance and differential structures coincide.   
\end{abstract}


\section{Introduction}

Let $\Omega\subset\bR^n$, $n\geq 2$, be a domain and $F$ an admissible Finsler structure on $\Omega$ (the precise definition is given in Section~2 below). Associated to $F$, we have the following intrinsic distance defined by
\begin{align}\label{def:delta}
\delta_F(x,y)=\sup_{u}\big\{u(x)-u(y): u \text{ is Lipschitz and } \|F(x, du(x))\|_\infty\leq 1\big\}.
\end{align}
Above, $du(x)$ denotes the differential of the Lipschitz function $u$ at a point $x$. Recall that the well-known Rademacher's theorem implies that $du(x)$ exists at almost every $x\in \Omega$ and thus the above definition makes sense.  
The elliplicity condition on $F$ implies that $\delta_F$ is locally comparable to the standard Euclidean distance. We define the pointwise Lipschitz constant of a Lipschitz function $u:\Omega\to \bR$ by setting
\begin{align*}
\Lip_{\delta_F}u(x)=\limsup_{y\to x}\frac{|u(y)-u(x)|}{\delta_F(x,y)}.
\end{align*}
Given a subset $K$ of $\bR^n$, we set 
\begin{align*}
\Lip_{\delta_F}(u,K)=\sup_{x,y\in K,x\neq y}\frac{|u(x)-u(y)|}{\delta_F(x,y)}
\end{align*}
and denote by $\Lip_{\delta_F}(K)$ the collection of all functions $u:K\to\bR$ with $\Lip_{\delta_F}(u,K)<\infty$. 

Sturm asked the following interesting question in~\cite{s97}: is a diffusion process determined by the intrinsic distance? Mathematically, Sturm's question can be formulated as follows: is it true that for all $u\in \Lip_{\delta_F}(\Omega)$,
\begin{align*}
F(x, du(x))=\Lip_{\delta_F}u(x)
\end{align*}
almost everywhere with $F(x,v)=\sqrt{\langle A(x)v,v\rangle}$? 

The answer to the question is yes when $A$ is supposed to be continuous, as shown by Sturm in~\cite[Proposition 4]{s97}. He also pointed out that the answer to this question is not always positive~\cite[Theorem 2]{s97}: for $F(x,v)=\sqrt{\langle A(x)v,v\rangle}$, where $A$ is a diffusion matrix, there exists $\tilde{F}(x,v)= \sqrt{\langle\tilde{A}(x)v,v\rangle}$ such that $\delta_F=\delta_{\tilde{F}}$ but 
\begin{align*}
F(x,v)<\tilde{F}(x,v)
\end{align*}
for all $v\in \bR^n\backslash \{0\}$; see also~\cite{kz12} for a different example.

The case $F(x,v)=\sqrt{\langle A(x)v,v\rangle}$ gained deeper understanding in a recent paper~\cite{ksz14}, where the authors enhanced Sturm's result by showing that if the diffusion matrix $A$ is weak upper semicontinuous, then the differential and distance structures coincide. They also constructed an example, which shows that if $A$ fails to be upper semicontinuous on a set of positive measure, then the differential and distance structure may fail to coincide.

The main purpose of this paper is to generalize the above result of~\cite{ksz14} to more general Finsler structures. More precisely, we are going to prove the following result.

\begin{theorem}\label{thm:coincide of differential and metric}
Let $n\geq 2$ and $F$ be an admissible Finsler structure on a domain $\Omega\subset \bR^n$. If $F$ is weak upper semicontinuous on $\Omega$, then the intrinsic distance and differential structure coincide. That is given a Lipschitz function $u$ on $\Omega$ (with respect to the Euclidean distance), for almost every $x\in \Omega$, we have
\begin{align*}
\Lip_{\delta_F}u(x)=F(x, du(x)).
\end{align*}
\end{theorem}

The proof of~\cite[Theorem 2]{ksz14} relies heavily on the structure of  $F(x,v)=\sqrt{\langle A(x)v,v\rangle}$. It seems that there is little hope to adapt their proofs in the greater generality of this paper. 

To see an example where Theorem~\ref{thm:coincide of differential and metric} applies more generally than~\cite[Theorem 2]{ksz14}, we may choose suitable weighted $L^p$-norm with $1\leq p<\infty$. For instance, consider $F(x,v)=\big(\sum_{i=1}^n w(x)|v_i|^p\big)^{1/p}$, where the weight function $w$ is upper semicontinuous and satisfies the ellipticity condition $0<c\leq w(x)\leq C<\infty$ for all $x\in \bR^n$.

Theorem~\ref{thm:coincide of differential and metric} can be regarded as an improved version of~\cite[Proposition 2.4]{gpp06} from $L^\infty$-norm to pointwise equality.

Our proof of Theorem~\ref{thm:coincide of differential and metric} completely differs from that used in~\cite{ksz14} and it is simpler than~\cite{ksz14}, even in their setting. The crucial observation is Proposition 3.1 below, a special case of a result due to De Cecco and Palmieri~\cite{dp93}, which states that the intrinsic distance $\delta_F$ (infinitesimally) coincides with $d_c^*$, where $d_c^*$ is the distance induced by the Finsler structure $F$. The weak upper semicontinuity is crucial for our proof, since it implies that the ``metric density" of a curve with respect to the metric length coincides with its ``differential density"; see Section 4 below for the precise meaning. 
Our approach is more geometric and was influented a lot by the recent studies in Finsler geometry~\cite{dp93,dp95,bcs00,bbi01}. Some of the ideas from this paper were successfully used in our companion paper~\cite{gky15} on certain $L^\infty$-variational problems associated to measurable Finsler structures. It is known (e.g.~\cite{ags14,kz12}) that the intrinsic distance and differential structures coincide even for abstract Dirichlet forms on metric measure spaces. It would be interesting to know that whether a verion of Theorem~\ref{thm:coincide of differential and metric} holds in the abstract setting as there.

This paper is organized as follows. Section 2 contains all the preliminaries related to Finsler structures. Section 3 and Section 4 contain an overview of the necessary background that are needed for our proof of Theorem~\ref{thm:coincide of differential and metric}. In Section 5, we prove Theorem~\ref{thm:coincide of differential and metric}. The appendix contains a separate proof of Proposition~\ref{prop:key proposition} under the weak upper semicontinuity assumption.

\section{Preliminaries on Finsler structures}
Let $\Omega\subset\bR^n$, $n\geq 2$, be a domain, i.e., an open connected set.

\begin{definition}[Finsler structures]\label{def:finsler structure}
	We say that a function $F:\Omega\times \bR^n\to [0,\infty)$ is a Finsler structure on $\Omega$ if 
	\begin{itemize}
		\item $F(\cdot,v)$ is Borel measurable for all $v\in \bR^n$, $F(x,\cdot)$ is continuous for a.e. $x\in \Omega$;
		\item $F(x,v)>0$ for a.e. $x$ if $v\neq 0$;
		\item $F(x,\lambda v)=|\lambda| F(x,v)$ for a.e. $x\in \Omega$ and for all $\lambda\in \bR$ and $v\in \bR^n$.
	\end{itemize}
\end{definition}

\begin{definition}[Admissible Finsler structures]\label{def:admissible finsler stru}
	A Finsler structure $F$ is said to be admissible if 
	\begin{itemize}
		\item $F(x,\cdot)$ is convex for a.e. $x\in \Omega$; 
		\item $F$ is locally equivalent to the Euclidean norm or elliptic, i.e.,  there exists a continuous function $\lambda:\Omega\to [1,\infty)$ such that
		\begin{align*}
			\frac{1}{\lambda(x)}|v|\leq F(x,v)\leq \lambda(x)|v|
		\end{align*}
		for a.e. $x\in \Omega$ and for all $v\in \bR^n$.
	\end{itemize}
\end{definition}

It is straightforward to verify that the standard $L^p$-norm ($1\leq p<\infty$), i.e., $F(x,v)=(\sum_{i=1}^nv_i^p)^{1/p}$, is an admissible Finsler structure on $\bR^n$. From the geometric point of view, there are many other interesting examples and we refer the interested readers to~\cite{bcs00} for the details.

Recall that a function $u:\Omega\to \bR$ is said to be upper semicontinuous at $x\in \Omega$ if 
\begin{align*}
	u(x)\geq \limsup_{y\to x}u(y).
\end{align*}
Following~\cite{ksz14}, we say that $u$ is weak upper semicontinuous in $\Omega$ if $u$ is upper semicontinuous at almost every $x\in \Omega$. 
Let $F$ be an admissible Finsler structure on $\Omega$. We say that $F$ is weak upper semicontinuous on $\Omega$ if for each $v\in \bR^n$, the function $F(\cdot,v)$ is weak upper semicontinuous on $\Omega$.

Similarly a function $u:\Omega\to \bR$ is said to be lower semicontinuous at $x\in \Omega$ if 
\begin{align*}
	u(x)\leq \liminf_{y\to x}u(y),
\end{align*}
and $u$ is weak lower semicontinuous in $\Omega$ if $u$ is lower semicontinuous at almost every $x\in \Omega$. 
Let $F$ be an admissible Finsler structure on $\Omega$. We say that $F$ is weak lower semicontinuous on $\Omega$ if for each $v\in \bR^n$, the function $F(\cdot,v)$ is weak lower semicontinuous on $\Omega$.

Let $F$ be an admissible Finsler structure for $\Omega$. We introduce the dual of $F:\Omega\times\bR^n\to [0,\infty)$ in the standard way.

\begin{definition}[Dual Finsler structures]\label{def:dual finsler}
The dual $F^*$ of an admissible Finsler structure $F:\Omega\times\bR^n\to [0,\infty)$ is defined as
\begin{align*}
F^*(x,w)&=\sup_{v\in \bR^n}\Big\{\langle v,w\rangle: F(x,v)\leq 1\Big\}\\
&=\max_{v\neq 0}\Big\{\langle w,\frac{v}{F(x,v)}\rangle \Big\},
\end{align*}
where $\langle \cdot,\cdot\rangle$ is the standard inner product in $\bR^n$.
\end{definition}

The following proposition follows immediately from Definition~\ref{def:dual finsler}; see for instance~\cite[Section 1.2]{gpp06} or~\cite[Section 2]{bd05} for more information.

\begin{proposition}[Basic properties of a dual Finsler structure]\label{prop:basic prop of dual} Let $F$ be an admissible Finsler structure on $\Omega$. Then the dual function $F^*$ satisfies the following properties
\begin{itemize}
\item $F^*(\cdot,v)$ is Borel measurable and $F^*(x,\cdot)$ is Lipschitz;
\item $F^*(x,\cdot)$ is a norm;
\item $F^*(x,\cdot)$ is locally equivalent to the Euclidean norm, \ie 
\begin{align*}
\frac{1}{\lambda(x)}|v|\leq F^*(x,v)\leq \lambda(x)|v|.
\end{align*}
\item $(F^*)^*(x,v)= F(x,v)$;
\item $F$ is weak upper (lower) semicontinuous if and only if $F^*$ is weak lower (upper) semicontinuous.
\end{itemize}
\end{proposition}

\section{Comparison of intrinsic distances}
Let $(\Omega,F(x,\cdot),d_c^F,\delta_F)$ be a Finsler manifold with an admissible Finsler structure $F$. 
For an admissible Finsler structure $F$ on $\Omega$, we may associate a distance in the standard way by setting

\begin{align*}
d_c^*(x,y)&:=\sup_{N}\inf_{\gamma\in \Gamma_N^{x,y}}\Big\{\int_0^1 F^*(\gamma(t),\gamma'(t))dt\Big\},
\end{align*}
where the supremum is taken over all subsets $N$ of $\Omega$ such that $|N|=0$ and $\Gamma_N^{x,y}(\Omega)$ denotes the set of all Lipschitz curves in $\Omega$ with end points $x$ and $y$ transversal to $N$,\ie such that $\mathscr{H}^1(N\cap \gamma)=0$. For an admissible Finsler metric $F$, $d_c^*$ is indeed an intrinsic distance; for the definition of an intrinsic distance and this fact, see~\cite{dp93,dp95}. Above, we use $|E|$ to denote the $n$-dimensional Lebesgue measure of a set $E\subset \bR^n$ and $\mathcal{H}^1$ the one-dimensional Hausdorff measure.

The following fundamental result, which relates $\delta_F$ and $d_c^*$, was a special case of~\cite[Theorem 3.7]{dp93}.
\begin{proposition}\label{prop:key proposition}
Let $F$ be an admissible Finsler structure on $\Omega$. Then for almost every $x\in \Omega$, it holds
\begin{align}\label{eq:infinitesimal coincidence}
\lim_{y\to x}\frac{\delta_F(x,y)}{d_c^*(x,y)}=1.
\end{align}
\end{proposition}

Since we have assumed the weak upper semicontinuity on our admissible Finsler structure in our main result Theorem~\ref{thm:coincide of differential and metric}, we give a separate proof of Proposition~\ref{prop:key proposition} under this extra assumption in the appendix.

\section{Comparison of metric derivatives}
For any distance $d$ on $\Omega$ and each Lipschitz (with respect to $d$) curve $\gamma:[a,b]\to \Omega$, the length of $\gamma$ with respect to $d$ is denoted by $\mathcal{L}_d(\gamma)$, i.e., 
\begin{align*}
\mathcal{L}_d(\gamma):=\sup\big\{\sum_{i=1}^k d(\gamma(t_i),\gamma(t_{i+1})) \big\},
\end{align*}
where the supremum is taken over all partitions $\{[t_i,t_{i+1}]\}$ of $[a,b]$. 

Given a curve $\gamma$, the metric derivative of $\gamma$ at $t$ is defined to be
\begin{align*}
|\gamma'(t)|_d:=\limsup_{s\to 0}\frac{d(\gamma(t+s),\gamma(t))}{s}.
\end{align*}
If $\gamma:[a,b]\to \Omega$ is Lipschitz with respect to $d$, then its length can be computed by integrating the metric derivative, \ie
\begin{align*}
\mathcal{L}_d(\gamma)=\int_a^b|\gamma'(t)|_d dt.
\end{align*}
In other words, for a Lipschitz curve, the metric derivative is the metric density of its length.

For any intrinsic distance $d$, which is locally bi-Lipschitz equivalent to the Euclidean distance, we may associate a Finsler structure $\Delta_d$ in the following manner. For each $x\in \Omega$ and for every direction $v$, we define
\begin{align}\label{def:new finsler distance}
\Delta_d(x,v):=\limsup_{t\to 0^+}\frac{d(x,x+tv)}{t}.
\end{align} 
It can be proved that for every Lipschitz curve $\gamma:[a,b]\to \Omega$, we have 
\begin{align*}
\mathcal{L}_d(\gamma)=\int_a^b\Delta_d(\gamma(t),\gamma'(t))dt.
\end{align*}
In particular, $\Delta_d(\gamma(t),\gamma'(t))=|\gamma'(t)|_{d}$ for a.e. $t\in [a,b]$.

\begin{remark}\label{rmk:for compare metric derivative with differential}
For any admissible Finsler structure $F$, one always has
\begin{align}\label{eq:nocoincideness of metric derivative and differential}
\Delta_{d_c^*}(x,v)\leq F^*(x,v)\quad \text{ for a.e. } x\in \Omega \text{ and all } v\in \bR^n;
\end{align}
see~\cite[Proposition 1.6]{gpp06}. However, the equality does not necessary hold; See~\cite[Example 5.1]{dp95} for a counter-example.
\end{remark}

In addition, for an admissible Finsler structure $F$, the dual Finsler structure $F^*$ always induces a lower semicontinuous length structure; see~\cite[Section 2.4.2]{bbi01}. Moreover, if the Finsler metric $F$ is weak upper semicontinuous on $\Omega$, then the following stronger result holds.
\begin{proposition}[Proposition 2.9,~\cite{bd05}]\label{prop:metric derivative coincide with density}
If the Finsler structure $F$ is weak upper semicontinuous on $\Omega$, then for a.e. $x\in \Omega$ and all $v\in \bR^n$, it holds
\begin{align*}
\Delta_{d_c^*}(x,v)=F^*(x,v).
\end{align*}
\end{proposition}

\section{Coincidence of distance structure and differential structure}
In this section, we are ready to prove our main result Theorem~\ref{thm:coincide of differential and metric}. 

\begin{proposition}\label{prop:trivial part}
For each $u\in \Lip_{\delta_F}(\Omega)$, $F(x, du(x))\leq \Lip_{\delta_F}u(x)$ for a.e. $x\in \Omega$.
\end{proposition}
\begin{proof}
Since both sides are positively 1-homogeneous with respect to $u$, we only need to show that for a.e. $x\in \Omega$, if $\Lip_{\delta_F}u(x)= 1$, then $F(x, du(x))\leq 1$. 

Note that by Proposition~\ref{prop:key proposition}, for a.e. $x\in \Omega$, $\Lip_{\delta_F}u(x)=\Lip_{d_c^*}u(x)$. Fix such an $x$. For each $v\in \bR^n$, we have 
\begin{align*}
du(x)v&=\lim_{t\to 0}\frac{u(x+tv)-u(x)}{t}\\
&\leq \limsup_{t\to 0}\frac{d_c^*(x,x+tv)}{t}\cdot\limsup_{t\to 0}\frac{u(x+tv)-u(x)}{d_c^*(x,x+tv)}\\
&\leq \Delta_{d_c^*}(x,v)\Lip_{d_c^*}u(x)\leq F^*(x,v),
\end{align*}
where in the last inequality, we have used the inequality~\eqref{eq:nocoincideness of metric derivative and differential}.

Therefore, 
\begin{align*}
F(x, du(x))&=F^{**}(x,du(x))\\
&=\max_{v\neq 0} \Big\{du(x)\Big(\frac{v}{F^*(x,v)}\Big)\Big\}\leq 1
\end{align*}
as desired. This completes our proof.
\end{proof}

\begin{theorem}\label{thm:main result}
Let $F$ be an admissible Finsler structure on $\Omega$. If $F$ is weak upper semicontinuous on $\Omega$, then for any Lipschitz function $u$ in $(\Omega,\delta_F)$, $F(\cdot, du(\cdot))$ is an upper gradient of $u$. In particular, this implies that 
$$\Lip_{\delta_F}u(x)\leq F(x, du(x))$$
for a.e. $x\in \Omega$. 
\end{theorem}

\begin{proof}
	First, note that our assumption on $F$ implies that $F$ satisfies the following uniform upper semicontinuity property, for a.e. $x\in \Omega$,
	\begin{align}\label{eq:referee equation}
		\forall \varepsilon>0, \, \exists\, \delta>0:\, F(y,v)\leq (1+\varepsilon) F(x,v)\quad \text{for all } y\in B(x,\delta),\, v\in \bR^n.
	\end{align}
	By homogeneity of $F$ (with respect to $v$), it suffices to prove~\eqref{eq:referee equation} for all $v\in \mathbb{S}$ (the unit sphere). Suppose by contradiction, that~\eqref{eq:referee equation} fails. Then there exist some $x\in \Omega$ and some $\varepsilon_0>0$ such that for each $k\in \mathbb{N}$, there exist some $y_k\in B(x,\frac{1}{k})$ and $v_k\in \mathbb{S}$  so that 
	\begin{align}
	F(y_k,v_k)>(1+\varepsilon_0)F(x,v_k).	
	\end{align}
	By compactness of $\mathbb{S}$, we may assume (up to another subsequence if necessary) $v_k\to v\in \mathbb{S}$ as $k\to \infty$. Then
	\begin{align*}
		F(x,v)&=\limsup_{k\to \infty}F(x,v_k)\geq \limsup_{k\to \infty}\limsup_{y\to x}F(y,v_k)\\
		&\geq \limsup_{k\to \infty}F(y_k,v_k)\geq \limsup_{k\to \infty}(1+\varepsilon_0)F(x,v_k)\\
		&=(1+\varepsilon_0)F(x,v),
	\end{align*}
	which is a contradiction. 
	
	Secondly, by Rademacher's Theorem, it suffices to prove Theorem~\ref{thm:main result} when $u(x)=\langle v,x\rangle$ is linear. We may additionally assume that $v\neq 0$. By the fundamental theorem of calculus and the definition of $F^*$, we have
	\begin{align*}
	|u(x)-u(y)|&=|\langle v,y-x\rangle|=\Big|\int_0^1 \frac{d}{dt}u(\gamma(t))dt\Big|\\
	&=\Big|\int_0^1\langle v,\gamma'(t)\rangle dt\Big|\leq (1+\varepsilon)F(x,v)\int_0^1 F^*(\gamma(t),\gamma'(t))dt	
	\end{align*}
	whenever $x, y$ and $\gamma(t)$ belongs to the ``$\delta$-neighborhood of $x$ where~\eqref{eq:referee equation} holds; it follows	that
	\begin{align*}
		\frac{|\langle v,y-x\rangle|}{d_c^*(x,y)}\leq (1+\varepsilon)F(x,v),
	\end{align*}
	whenever $|x-y|<\delta$. Letting $y\to x$ and $\varepsilon\to 0$ concludes our proof.
	
\end{proof}

\section*{Appendix: Proof of Proposition~\ref{prop:key proposition} when $F$ is weak upper semicontinuous}

\begin{proof}

	The inequality $\delta_F(x,y)\leq d_c^*(x,y)$ follows directly from definitions. Indeed, for each Lipschitz function $u$ with $\|F(\cdot, du(\cdot))\|_{L^{\infty}(\Omega)}\leq 1$, each $x,y\in \Omega$, for each Lipschitz curve $\gamma$ joining $x$ and $y$ that is transversal to the zero measure set $N=\{x\in \Omega:F(x, du(x))> 1\}$,
	\begin{align*}
		u(x)-u(y)&=\int_0^1  du(\gamma(t))\big(\gamma'(t)\big) dt\\
		&\leq \int_0^1 F^*(\gamma(t),\gamma'(t))dt=\mathcal{L}_{d_c^*}(\gamma),
	\end{align*}
	where $\mathcal{L}_{d_c^*}$ denotes the length of the curve $\gamma$ with respect to the metric $d_c^*$.	Taking infimum over all admissible curves on the right-hand side and then supermum over all admissible functions over the left-hand side, we obtain via Proposition~\ref{prop:metric derivative coincide with density} that 
	\begin{align*}
		\delta_F(x,y)\leq d_c^*(x,y).
	\end{align*} 
	In particular, 
	\begin{align*}
		\limsup_{y\to x}\frac{\delta_F(x,y)}{d_c^*(x,y)}\leq 1.
	\end{align*}
	
	We are left to prove that
	\begin{align}\label{eq:liminf ineq}
		\liminf_{y\to x}\frac{\delta_F(x,y)}{d_c^*(x,y)}\geq 1.
	\end{align}
	We divide the proof of this equation into two steps.
	
	\textit{Step 1:} assume that $F(\cdot,v)$ is continuous.
	
	Fix $x\in \Omega$ and $\varepsilon>0$. Since $F(\cdot,v)$ and $F^*(\cdot,v)$ are continuous in $B(x,\delta)$, we may assume that for all $z\in B(x,\delta)$, 
	\begin{align*}
		(1-\varepsilon)F(z,v)\leq F(x,v)\leq (1+\varepsilon)F(z,v)
	\end{align*}
	and 
	\begin{align*}
		(1-\varepsilon)F^*(z,v)\leq F^*(x,v)\leq (1+\varepsilon)F^*(z,v).
	\end{align*}
	Note that the issue is local, we are now restricting ourself to the ball $B(x,\delta)$.
	
	Consider the curve $\gamma(t)=x+t(y-x)$, we have
	\begin{align*}
		d_c^*(x,y)\leq \mathcal{L}_{d_c^*}(\gamma)=\int_0^1F^*(\gamma(t),\gamma'(t))dt\leq (1+\varepsilon)F^*(x,y-x).
	\end{align*}
	By the definition of a dual Finsler structure, we know that there exists some $\tilde{v}\neq 0$ such that $F^*(x,y-x)=\langle y-x,\frac{\tilde{v}}{F(x,\tilde{v})}\rangle$. Set 
	$$v:=\frac{\tilde{v}}{(1+\varepsilon)F(x,\tilde{v})}.$$
	Then $F(x,v)=\frac{1}{1+\varepsilon}$ and $\langle v,y-x\rangle=\frac{1}{1+\varepsilon}F^*(x,y-x)$. Note that for all $z\in B(x,\delta)$,
	$F(z,v)\leq (1+\varepsilon)F(x,v)\leq 1$ and so the function $u(z):=\langle v,z\rangle$ is an admissible function for $\delta_F(x,y)$. This means that
	\begin{align*}
		\delta_F(x,y)\geq u(y)-u(x)=1/(1+\varepsilon)F^*(x,y-x)\geq \frac{1}{(1+\varepsilon)^2}d_c^*(x,y).
	\end{align*}
	It is clear that~\eqref{eq:liminf ineq} follows from the above inequality by letting $\varepsilon\to 0$. 
	
	\textit{Step 2:} assume that $F(\cdot,v)$ is weak upper semicontinuous.
	
	In this case, $F^*$ is weak lower semicontinuous, it is a well-known fact that there exists a sequence of admissible Finsler norms $F_n^*(\cdot,v)$, which is continuous in the first variable, such that 
	\begin{align*}
		F_n(x,v)^*\leq F_{n+1}^*(x,v)\leq \cdots \to F^*(x,v);
	\end{align*}
	and $d_c^{*n}\to d_c^*$ as $n\to \infty$, where $d_c^{*n}$ is the distance induced by the Finsler structure $F_n$; see for instance~\cite[Section 4]{d05}. Let $F_n=F_n^{**}$ denote the dual of $F_n^*$, then it is easy to check from our definition that 
	\begin{align*}
		F_n(x,v)\geq F_{n+1}(x,v)\geq \cdots\to F(x,v).
	\end{align*} 
	It follows that
	\begin{align*}
		\frac{\delta_F(x,y)}{d_c^*(x,y)}=\lim_{n\to \infty}\frac{\delta_{F_n}(x,y)}{d_c^{*n}(x,y)},
	\end{align*}
	where $\delta_{F_n}$ is the intrinsic distance induced by $F_n$ similar as $\delta_{F}$. Given $\varepsilon>0$, there exists $N_0$ such that for all $n\geq N_0$, 
	\begin{align*}
		\frac{\delta_F(x,y)}{d_c^*(x,y)}\geq (1-\varepsilon)\frac{\delta_{F_n}(x,y)}{d_c^{*n}(x,y)}.
	\end{align*}
	On the other hand, by step 1, 
	\begin{align*}
		\liminf_{y\to x}\frac{\delta_{F_n}(x,y)}{d_c^{*n}(x,y)}\geq 1.
	\end{align*}
	We thus obtain
	\begin{align*}
		\liminf_{y\to x}\frac{\delta_{F_n}(x,y)}{d_c^{*n}(x,y)}\geq \liminf_{y\to x}(1-\varepsilon)\frac{\delta_{F_n}(x,y)}{d_c^{*n}(x,y)}\geq 1-\varepsilon.
	\end{align*}
	The claim follows by letting $\varepsilon\to 0$.

\end{proof}

\textbf{Acknowledgements}

The author would like to thank Professor Pekka Koskela, Professor Yuan Zhou and Dr. Changlin Xiang for helpful discussions. He is also very grateful to Professor Luigi Ambrosio, Professor Andrea Davini and Professor Giuliana Palmieri for their interests in this work. In particular, he is grateful to Professor Giuliana Palmieri, who pointed out a mistake in an earlier version of this paper. Finally, he would like to thank the anonymous referees for their insightful comments that greatly increased the readability of the paper.


\begin{thebibliography}{99}
\bibitem{ags14}
L. Ambrosio, N. Gigli and G. Savar\'e, \textit{Metric measure spaces with Riemannian Ricci curvature bounded from below}, Duke Math. J. 163 (2014), no. 7, 1405-1490.


\bibitem{bcs00}
D. Bao, S.S. Chern and Z. Shen, \textit{An introduction to Riemann-Finsler geometry}. Graduate Texts in Mathematics, 200. Springer-Verlag, New York, 2000.

\bibitem{bd05}
A. Briani and A. Davini, \textit{Monge solutions for discontinuous Hamiltonians}, ESAIM Control Optim. Calc. Var. 11 (2005), no. 2, 229-251

\bibitem{bbi01}
D. Burago, Y. Burago and S. Ivanov, \textit{A course in metric geometry}, Graduate Studies in Mathematics, 33. American Mathematical Society, Providence, RI, 2001. 


\bibitem{d05}
A. Davini, \textit{Smooth approximation of weak Finsler metrics}, Differential Integral Equations 18 (2005), no. 5, 509-530.

\bibitem{dp93}
G. De Cecco and G. Palmieri, \textit{Intrinsic distance on a LIP Finslerian manifold}, (Italian) Rend. Accad. Naz. Sci. XL Mem. Mat. (5) 17 (1993), 129-151.

\bibitem{dp95}
G. De Cecco and G. Palmieri, \textit{LIP manifolds: from metric to Finslerian structure}, Math. Z. 218 (1995), no. 2, 223-237. 
 
\bibitem{gpp06}
A. Garroni, M. Ponsiglione and F. Prinari, \textit{From 1-homogeneous supremal functionals to difference quotients: relaxation and $\Gamma$-convergence}, Calc. Var. Partial Differential Equations 27 (2006), no. 4, 397-420. 

\bibitem{gky15}
C.-Y. Guo, C.-L. Xiang and D. Yang, \textit{The $L^\infty$-variational problems associated to measurable Finsler structures}, Nonlinear Anal. 132 (2016), 126-140.



\bibitem{ksz14}
P. Koskela, N. Shanmugalingam and Y. Zhou, \textit{Intrinsic geometry and analysis of diffusion processes and $L^\infty$-variational problems}, Arch. Ration. Mech. Anal. 214 (2014), no. 1, 99-142.


\bibitem{kz12}
P. Koskela and Y. Zhou, \textit{Geometry and analysis of Dirichlet forms}, Adv. Math. 231 (2012), no. 5, 2755-2801.
 

\bibitem{s97}
K.T. Sturm, \textit{Is a diffusion process determined by its intrinsic metric?} Chaos Solitons Fractals 8 (1997), no. 11, 1855-1860.


\end{thebibliography}
\end{document}